\documentclass[12pt,a4paper]{amsart}
\usepackage{graphics}
\usepackage{epsfig}
\usepackage{graphicx}
\theoremstyle{plain}
\usepackage{amssymb}
\usepackage[ukrainian,english]{babel}
\advance\hoffset-20mm \advance\textwidth40mm

\newtheorem{theorem}{Theorem}
\newtheorem{lemma}{Lemma}
\newtheorem*{theo*}{Theorem}

\newtheorem{corollary}{Corollary}
\theoremstyle{definition}

\newtheorem*{definition*}{Definition}

\newtheorem{remark}{Remark}

\DeclareMathOperator{\Ker}{Ker}

\DeclareMathOperator{\Der}{Der}
\DeclareMathOperator{\diver}{div}

\begin{document}
\sloppy
\title[On Lie algebras consisting of locally nilpotent  derivations]
{On Lie algebras consisting of locally nilpotent  derivations}
\author
{A.P.  Petravchuk and K. Ya. Sysak}
\address{Anatoliy P. Petravchuk:
Department of Algebra and Mathematical Logic, Faculty of Mechanics and Mathematics, 
Taras Shevchenko National University of Kyiv, 64, Volodymyrska street, 01033  Kyiv, Ukraine}
\email{aptr@univ.kiev.ua , apetrav@gmail.com}
\address{Kateryna Ya. Sysak:
Department of Algebra and Mathematical Logic, Faculty of Mechanics and Mathematics, Taras Shevchenko National University of Kyiv, 64, Volodymyrska street, 01033  Kyiv, Ukraine}
\email{sysakkya@gmail.com}
\date{\today}

%\thanks{
% The second author was partially supported by the DFFD,
%grant F28.1/026}
%
\begin{abstract}
Let $K$ be an algebraically closed   field of characteristic zero and $A$  an integral $K$-domain. The Lie algebra $Der _{K}(A)$ of all $K$-derivations of $A$ contains the set $LND(A)$ of all locally nilpotent derivations. The structure of $LND(A)$ is of great interest, and the question about properties of Lie algebras contained in $LND(A)$ is still open. An answer to it in the finite dimensional case is given. It is proved that any finite dimensional (over $K$) subalgebra of $Der _{K}(A)$ consisting of locally nilpotent derivations is nilpotent. In the case $A=K[x, y],$ it is also proved that  any  subalgebra of $Der _{K}(A)$ consisting of locally nilpotent derivations is conjugated by an automorphism of $K[x, y]$ with a subalgebra of the triangular Lie algebra.
 \end{abstract}
\maketitle

%%%%%%%%%%%%%%%%%%%%%%%%%%%%%%%%%%%%%%%%%%%%%%%%%%%%%%%%%%%

\section{Introduction}

Let $\mathbb K$ be an algebraically closed  field of characteristic zero and $A$ an associative commutative $\mathbb K$-algebra that is a domain. A derivation $D:A\to A$ is called locally nilpotent if for any element $a\in A$ there exists a positive integer $n=n(a)$ such that $D^{n}(a)=0.$  The study of  locally nilpotent derivations is an important problem in differential algebra because the exponents of such  derivations are  automorphisms of the associative algebra $A$ and they carry information about $A.$ Many papers and a few monographs are devoted to locally nilpotent derivations (see, for example, \cite{Rentschler}, \cite{Fr}, \cite{ML}, \cite{Nowicki}, \cite{Essen}, \cite{Kaliman}, etc). One of  unsolved problems is to describe all Lie algebras contained in the set ${\rm LND}(A)$ of all locally nilpotent derivations on the algebra $A$  (see Problem 11.6  in  \cite{Fr}). In this paper, it is proved  that every finite dimensional (over $\mathbb K$)  subalgebra of the Lie algebra $\Der_{\mathbb K}(A)$  consisting of locally nilpotent derivations is nilpotent (Theorem \ref{Th1}). In the case $A=\mathbb K[x, y]$, the polynomial ring in two variables, it is  proved that every subalgebra $L\subseteq {\rm LND}(A)$ of  ${\rm Der} _{\mathbb K}(A)$  is conjugated  with a subalgebra of the triangular Lie algebra $u_2(\mathbb K)$ by an automorphism of $\mathbb K[x, y].$ By Rentschler's Theorem \cite{Rentschler}, the structure of ${\rm LND}(\mathbb K[x, y])$ is as follows: $${\rm LND}(\mathbb K[x, y])=\bigcup _{\theta \in Aut(\mathbb K[x, y])} \theta u_{2}(\mathbb K)\theta ^{-1}.$$ It is proved that every Lie algebra lying in ${\rm LND}(\mathbb K[x, y])$ is contained entirely in at least one of subalgebras conjugated with $u_{2}(\mathbb K)$  (Theorem \ref{Th2}).

We use standard notations. The ground field $\mathbb K$ is algebraically closed of characteristic zero. The quotient field of the integral domain $A$ under consideration is denoted by $R.$ The set of all locally nilpotent derivations of $A$ is denoted by ${\rm LND}(A).$  Any derivation $D$ of $A$ can be uniquely extended  to a derivation of $R$ by the rule: $D(a/b)=(D(a)b-aD(b))/b^{2}.$
If $F$ is a subfield of the field $R$ and $r_{1}, \ldots , r_{k}\in R, $ then the set of all linear combinations of these elements with coefficients in $F$  is denoted by $F\langle r_{1}, \ldots , r_{k}\rangle ; $ it is a  subspace of the $F$-space $R.$
The triangular subalgebra $u{_2}(\mathbb K)$ of the Lie algebra $W_{2}(\mathbb K)=\Der (\mathbb K[x, y])$  consists of all the derivations on the ring $\mathbb K[x, y]$ of the form $D=\alpha \frac{\partial}{\partial x}+\beta (x)\frac{\partial}{\partial y},$  where $\alpha \in \mathbb K, \beta (x)\in \mathbb K[x]$ (about properties of triangular Lie algebras see \cite{Bavula1}).   Recall that a subalgebra $B$ of an associative commutative algebra $A$ is factorially closed in $A$ if the relations $a_1a_2\in B, a_1\not= 0, a_2\not =0,$ imply  $a_1\in B$ and $a_2\in B. $ A polynomial $a=a(x, y)\in \mathbb K[x, y]$ is called coordinate if there exists a polynomial $b=b(x, y)\in \mathbb K[x, y]$ such that $\mathbb K[x, y]=\mathbb K[a, b].$ Then the polynomials $a$ and $b$  form a coordinate pair $(a, b).$   If $f\in \mathbb K[x, y],$ then $f$ induces the  Jacobian derivation $D_{f}$ of the ring $\mathbb K[x, y]$ by the rule: $D_f(h)=\det J(f, h)$  for any $h\in \mathbb K[x, y],$ where $\det J(f, h)$ is the Jacobian determinant of the polynomials $f$ and $h.$ For any derivation $D=f\frac{\partial}{\partial x}+g\frac{\partial}{\partial y}\in \Der(\mathbb K[x,y])$ we denote by $\diver D$ the divergence of $D$: $\diver D=\frac{\partial f}{\partial x}+\frac{\partial g}{\partial y}.$ The Lie algebra $W_2(\mathbb K)$ is a free module over the ring $\mathbb K[x, y]$ (of rank $2$), so for any subalgebra $L\subseteq W_2(\mathbb K)$ one can define rank of $L$ over the ring $\mathbb K[x, y].$

\section{Finite dimensional Lie algebras consisting of locally nilpotent derivations}
Throughout this section,  $A$ denotes an integral $\mathbb K$-domain and $R$ the field of fractions for the algebra $A.$
The set  of all $\mathbb K$-derivations of $A$ is denoted by $\Der_{\mathbb K} A,$  it is a Lie algebra  over the field $\mathbb K.$ Some properties of locally nilpotent derivations are pointed out in the next two lemmas.
\begin{lemma}\label{principle_1}
Let  $D$ be a locally nilpotent derivation of the algebra   $A$ and  $\delta$  its extension   on the fraction field $R$ of  $A.$ Then:
\begin{enumerate}
\item[\rm(a)]{\cite[Principle 1]{Fr}}  $\Ker D$ is a factorially closed subring of  $A$.
\item[\rm(b)]\label{cor_1.23}{\cite[Corollary 1.23]{Fr}} $\Ker\delta=\mathrm{Frac}(\Ker D)$.
\item[\rm(c)]{\cite[Principle 11(e)]{Fr}} Transcendence degree of the field $R$ over the subfield $\Ker\delta$ equals $1$.
\item[\rm(d)]\label{cor_Fr}{\cite[Corollary 1.20]{Fr}} If $D(a)=ab$ for some  $a,b\in A$, then  $D(a)=0$.
\end{enumerate}
\end{lemma}

\begin{lemma}\label{principle_12}{\cite[Principle 12]{Fr}}
 Let $D_1, D_2\in L$ be locally nilpotent derivations of the algebra  $A$ such that  $\Ker D_1=\Ker D_2=B$. Then there exist nonzero elements  $a,b\in B$ such that  $aD_1=bD_2$.
\end{lemma}

\begin{lemma}\label{kryteriy}(see, for example, \cite{Jacob}, p.54).
Let  $L$ be a finite dimensional Lie algebra over an algebraically closed field  $\mathbb K$. Then  the algebra $L$ is nilpotent if and only if every two-dimensional subalgebra of $L$ is abelian.
\end{lemma}

\begin{lemma}\label{multyplying}
Let  $D_1,D_2\in \Der_{\mathbb K} A$ and  $a,b\in R$. Then
\begin{enumerate}
\item[\rm(a)] $[aD_1,bD_2]=ab[D_1,D_2]+aD_1(b)D_2-bD_2(a)D_1$.
\item[\rm(b)] If  $a,b\in \mathrm{Ker}D_1\cap \mathrm{Ker}D_2$, then  $[aD_1,bD_2]=ab[D_1,D_2]$.
\end{enumerate}
\end{lemma}
\begin{proof}
Straightforward check.
\end{proof}

\begin{theorem}\label{Th1}
 Let $\mathbb K$ be an algebraically closed field of characteristic zero and  $A$ an integral $\mathbb K$-domain.
 If $L$ is a finite dimensional  (over $\mathbb K$)  subalgebra of the Lie algebra  $\Der_{\mathbb K} A$  and every element of  $L$ is a locally nilpotent derivation on  $A$, then the Lie algebra  $L$ is nilpotent.
\end{theorem}

\begin{proof}
Let $M$ be any two-dimensional subalgebra    of the Lie algebra $L.$ Our aim is to  prove that  $M$ is abelian. Suppose this is not the case and  choose a basis   $\{ D_1,D_2\}$  of  the non-abelian subalgebra $M$ such that  $[D_1,D_2]=D_2.$   Let us   show that  $\Ker D_1\subseteq \Ker D_2$. Take any element  $f\in \Ker D_1$. Then  $$D_2(f)=[D_1,D_2](f)=(D_1D_2-D_2D_1)(f)=D_1(D_2(f))-D_2(D_1(f))=D_1(D_2(f)),$$
 because  $D_2(D_1(f))=D_2(0)=0$. Since  $D_1\in {\rm LND}(A)$ and  $D_1(D_2(f))=D_2(f)$, then $D_2(f)=0$ (by Lemma ~\ref{cor_Fr}(d)) and  $f\in \Ker D_2$. But then $\Ker D_1\subseteq \Ker D_2,$ because  the element  $f\in \Ker D_1$ is  arbitrarily chosen.

 Let $\delta_1$, $\delta_2$ be extensions of derivations  $D_1$ and $D_2$ respectively on the fraction field  $R=\mathrm{Frac}(A)$, and  let $R_1$, $R_2$ be subfields of constants for  $\delta_1$ and  $\delta_2$ respectively  in $R.$ The nonzero derivations $D_1$, $D_2$ are  locally nilpotent on  $A,$ so by   Lemma ~\ref{cor_1.23}(b) we get   equalities $\mathrm{Frac}(\Ker D_1)=R_1$ and  $\mathrm{Frac}(\Ker D_2)=R_2$.
 The inclusion  $\Ker D_1\subseteq \Ker D_2$ implies  $R_1\subseteq R_2$.
   Note that by Lemma ~\ref{principle_1}(a,c) the subfields $R_1,R_2$ are algebraically closed in the field $R$  and
${\rm tr.deg}_{R_1}R={\rm tr.deg}_{R_2}R=1.$ Then one can easily show  that $R_1=R_2$ and therefore  $\Ker D_1= \Ker D_2$.
Denote   $B=\Ker D_1=\Ker D_2.$   Using Lemma ~\ref{principle_12},  we see that there exist  nonzero elements  $a,b\in B$ such  that  $aD_1=bD_2$. But then we get   $$[aD_1,bD_2]=ab[D_1,D_2]=abD_2=0$$ by Lemma~\ref{multyplying}(b). Since  $D_2\neq 0$ we have  $ab=0.$ This is impossible, because $A$ is an integral domain. This contradiction shows that every two-dimensional subalgebra of the finite dimensional Lie algebra $L$ is abelian. Therefore, $L$ is nilpotent by Lemma \ref{kryteriy}.
\end{proof}
Recall that a Lie algebra $L$ over a field $\mathbb K$  is locally finite (or locally finite dimensional) if every its finitely generated subalgebra is of finite dimension  over $\mathbb K.$ A Lie algebra $L$ is locally nilpotent if every its finitely generated subalgebra is nilpotent.
\begin{corollary}
 Let $L$ be a locally  finite subalgebra of the Lie algebra  $\Der_{\mathbb K}(A).$ If $L\subseteq {\rm LND}(A),$ then the Lie algebra $L$ is locally nilpotent.
\end{corollary}

\section{On subalgebras of $W_2(\mathbb K)$ consisting of locally nilpotent derivations.}
In this section,  $A=\mathbb K[x,y]$ is the  polynomial ring in two variables over the field $\mathbb K$ and $R=\mathbb K(x,y),$ the  field of rational functions.  $W_2(\mathbb K)$ denotes the Lie algebra $\Der_{\mathbb K} A$ of all $\mathbb K$-derivations of $A$.

\begin{lemma}(\cite{Fr}, Corollary 4.7)\label{loc_nilp_der}
Let $D$ be a derivation of the ring $A=\mathbb K[x,y]$. Then $D$ is locally nilpotent if and only if $D=D_{f(a)}=f'(a)D_a$ for a coordinate polynomial $a\in A$ and some $f\in \mathbb K[t]$.
\end{lemma}

\begin{lemma}\label{product}
Let $D_f$, $D_g$ be Jacobian derivations of the ring $A=\mathbb K[x,y]$. Then $[D_f,D_g]=D_{[f,g]}$, where $[f,g]=\det J(f,g)$ is the Jacobian determinant  of polynomials $f,g\in A$.
\end{lemma}

\begin{proof}
Straightforward check.
\end{proof}

\begin{corollary}\label{Engel}
Let $L$ be a  subalgebra of the Lie algebra $W_2(\mathbb K)$. If $L\subseteq {\rm LND}(A),$  then $L$ satisfies the Engel condition, i.e. for any $D_1, D_2\in L$ there exists an integer $k\geq 1$ (depending on $D_1, D_2$) such that $[D_1,\underbrace{D_2,\dots,D_2}_{k}]=0$.
\end{corollary}
\begin{proof}
  Take arbitrary $D_1,D_2\in L$. Since $D_1,D_2\in {\rm LND}(A)$,  Lemma~\ref{loc_nilp_der} implies $D_1=D_f$ and $D_2=D_g$ for some $f,g\in A$. It follows from Lemma~\ref{product} that $[D_1,D_2]=D_{[f,g]}=-D_g(f),$  where $[f,g]=\det J(f,g)$ is the Jacobian determinant of $f,g\in A$. Further, $$[D_1,\underbrace{D_2,\dots,D_2}_{k}]=D_h,$$ where $h=[\dots[[f,\underbrace{g],g],\dots,g]}_{k}=[f,\underbrace{g,g,\dots,g}_{k}]$.
It is easy to check that $$[f,\underbrace{g,g,\dots,g}_{k}]=(-1)^kD_g^k(f).$$ Since $D_g$ is locally nilpotent, we get $D_g^k(f)=0$ for a sufficiently large $k$. The latter means that $[D_1,\underbrace{D_2,\dots,D_2}_{k}]=0$.
\end{proof}
\begin{lemma}\label{dependence}
Let $D_1$, $D_2$ be locally nilpotent derivations of the ring $A=\mathbb K[x,y]$.
\item[(1)]\label{lin_depend}If $D_1$ and $D_2$ are linearly dependent over $A$, then there exists a coordinate polynomial $a\in A$ such that $D_1=D_{f(a)}$, $D_2=D_{g(a)}$ for some $f,g\in \mathbb K[t]$.
\item[(2)]\label{coord_pair} If $D_1$ and $D_2$ are linearly independent over $A$ and $[D_1,D_2]=0$, then there exists a coordinate pair $(a,c)$ such that $D_1=D_a$, $D_2=D_c$.
\end{lemma}
\begin{proof}
Since $D_1\in {\rm LND} (A)$, Lemma~\ref{loc_nilp_der} implies that $D_1=D_{f(a)}$ for a coordinate pair $(a,b)$ and some $f\in \mathbb K[t]$.
Similarly, since $D_2\in {\rm LND}(A)$, there exists a coordinate pair $(c,d)$ such that $D_2=D_{g(c)}$ for some $g\in \mathbb K[t]$.
\item[(1)] Let $r_1D_1+r_2D_2=0$ for some $r_1,r_2\in A$, and at least one of $r_1,r_2$ is nonzero. Without loss of generality, one can assume that $D_1\neq 0$ and $D_2\neq 0$. Then it obviously holds the equality $\Ker D_1=\Ker D_2$. Since $\Ker D_1=\mathbb K[a]$ and $\Ker D_2=\mathbb K[c],$ we get $\mathbb K[a]=\mathbb K[c].$  Then $c=\varphi(a)$ for some $\varphi \in \mathbb K[t]$, and $D_2=D_{g(c)}=D_{g(\varphi (a))}=D_{g_1(a)}$.
\item[(2)] Let $D_1, D_2\in {\rm LND} (A)$ be linearly independent over $A$ and $[D_1,D_2]=0$. By Lemma~\ref{product}, it holds
$$[D_1,D_2]=[D_{f(a)},D_{g(c)}]=D_{[f(a),g(c)]}=0,$$ where $[f(a),g(c)]=\det J(f(a),g(c))$. It is easy to check that $[f(a),g(c)]=f'(a)g'(c)[a,c].$ Then we get $$D_{[f(a),g(c)]}=D_{f'(a)g'(c)[a,c]}=0,$$ and hence  $f'(a)g'(c)[a,c]\in \mathbb K$.

Let us show that $f'(a)g'(c)[a,c]\in \mathbb K^*$. Indeed, if $f'(a)g'(c)[a,c]=0$ then $[a,c]=0$, because $f'(a)\neq 0$, $g'(c)\neq 0$ (note that $\deg f \geq 1$ and $\deg g \geq 1$). The equality $[a,c]=0$ implies that $D_a$ and $D_c$ are linearly dependent over $A$ (see \cite[Corollary 7.2.10]{Nowicki}). This contradicts our assumption.

Therefore, $f'(a)g'(c)[a,c]\in \mathbb K^*$ and especially $[a,c]\in \mathbb K^*$. Since $(a,b)$ is a coordinate pair, there exists a polynomial $p(u,v)\in \mathbb K[u,v]$ such that $c=p(a,b)$. It follows $$[a,c]=[a,p(a,b)]=\frac{\partial}{\partial b}(p(a,b))[a,b]\in \mathbb K^*.$$ Thus, $ \frac{\partial}{\partial b}(p(a,b)\in \mathbb K^*$. This implies that $c=p(a,b)=\mu b+q(a)$ for some $\mu\in \mathbb K^*$ and $q\in \mathbb K[t]$.
Since the polynomials $a$ and $\mu b+q(a)$ form a coordinate pair in $A=\mathbb K[x, y],$ we get that $(a, c)$ is also a coordinate pair in $A.$
Furthermore, from the relation $f'(a)g'(c)[a,c]\in \mathbb K^*$ we get  $\deg f=\deg g=1$. Write $f(t)=\alpha t+\beta$,
$g(t)=\gamma t+\delta$ for $\alpha, \beta, \gamma, \delta\in\mathbb K$, $\alpha,\gamma\neq 0$. Then $D_1=D_{\alpha a+\beta}$ and $D_2=D_{\gamma b+\delta}$. Clearly,  $\alpha a+\beta$ and $\gamma b+\delta$ form a coordinate pair in $A$. Without loss of generality, we may denote $\alpha a+\beta$ by $a$, $\gamma b+\delta$ by $c$, and get $D_1=D_a$, $D_2=D_c$, where $(a,c)$ is a coordinate pair in $A$.
\end{proof}
\begin{lemma}\label{abelian}
Let $L$ be an abelian subalgebra of the Lie algebra $W_2(\mathbb K)$ (not necessarily finite dimensional over $\mathbb K$). If $L\subseteq {\rm LND}(A),$ then  $L$ is one of the following algebras:
\item[(1)] $L=\mathbb K\big\langle \{f_i(a)D_a\}_{i\in I}\big\rangle$, where $\{ f_i(t)\in \mathbb K[t], i\in I\}$ is a finite or a countable infinite  set of polynomials that  are linearly independent over $\mathbb K,$  and $a\in A$ is a coordinate polynomial.
\item[(2)] $L=\mathbb K\langle D_a,D_b \rangle$, where $(a,b)$ is a coordinate pair in $A$.
\end{lemma}

\begin{proof}
Let ${\rm rk}_A L=1$ and $D\in L$ is nonzero. By Lemma~\ref{loc_nilp_der}, there exists a coordinate polynomial $a\in A$ such that $D=D_{f(a)}=f'(a)D_a$ for some $f\in \mathbb K[t]$. Take an arbitrary $D_1\in L$. Then $D_1$ and $D$ are linearly dependent over $A$. By Lemma~\ref{lin_depend}(1), there exists $g\in \mathbb K[t]$ such that $D_1=D_{g(a)}=g'(a)D_a$. Thus, $L\subseteq \mathbb K[a]D_a$. Since $\mathbb K[a]D_a$ has a countable basis over $\mathbb K$, we can find a finite or a countable infinite basis $\{f_i(a)D_a\}_{i\in I}$ of the Lie algebra $L$. We see that $L$ is of type 1).

Now let ${\rm rk}_A L=2$. Take arbitrary $D_1,D_2\in L$ that are  linearly independent over $A$. Since the Lie algebra $L$ is abelian, $[D_1,D_2]=0$. By Lemma~\ref{coord_pair}(2), there exists a coordinate pair $(a,b)\in A$ such that $D_1=D_a$, $D_2=D_b$. Then for every $D=fD_a+gD_b\in L$, where $f,g\in A$, we have $$0=[D_a, D]=D_a(f)D_a+D_a(g)D_b,$$ $$0=[D_b, D]=D_b(f)D_a+D_b(g)D_b.$$
Since $D_a$ and $D_b$ are linearly independent over $A$, we have $D_a(f)=D_a(g)=0$ and $D_b(f)=D_b(g)=0$. These equalities imply that $f,g\in \Ker D_a\cap \Ker D_b=\mathbb K$. Therefore, $L=\mathbb K\langle D_a,D_b \rangle$.
\end{proof}

\begin{lemma}\label{rank_two}
Let $L$ be a subalgebra of rank $2$ over $A$ of the Lie algebra $W_2(\mathbb K).$  If $L\subseteq {\rm LND}(A),$  then there exist linearly independent (over $A$) elements $D_1,D_2\in L$ such that $[D_1,D_2]=0$. Moreover, there exists a coordinate pair $(a,b)\in A$ such that $D_1=D_a, D_2=D_b.$
\end{lemma}
\begin{proof}
  Take any elements $D_1,D_2\in L$ that are linearly independent over $A$. Consider inductively defined elements $D_{k+1}=[D_k, D_1]\in L$ for $k\geq 2$. By Corollary \ref{Engel}, there exists the least number $s$, $s\geq 2$, such that $D_{s+1}=0$. If $D_1$ and $D_s$ are linearly independent over $A$, then we denote $D_s$ by $D_2$ and all is done. Assume that $D_1, D_s$ are linearly dependent over $A$. By Lemma~\ref{lin_depend}(1), there exists a coordinate polynomial $a\in A$ such that $D_1=D_{f(a)}$, $D_s=D_{h(a)}$ for some $f,h\in \mathbb K[t]$. Since $D_{s-1}\in {\rm LND}(A)$, Lemma~\ref{loc_nilp_der} implies that there exists a coordinate polynomial $c\in A$ such that  $D_{s-1}=D_{g(c)}$ for some $g\in \mathbb K[t]$.
Note that $D_1$ and $D_{s-1}$ are linearly independent over $A$. Indeed, in the opposite case $D_1=D_{f_1(d)}$ and $D_{s-1}=D_{g_1(d)}$ for some  coordinate polynomial $d\in A$ and $f_1,g_1\in \mathbb K[t]$ (see Lemma \ref{dependence}).  By Lemma~\ref{product}, $$[D_{g_1(d)},D_{f_1(d)}]=D_{g'_1(d)f'_1(d)[d,d]}=0,$$ and thus $[D_{s-1},D_1]=D_s=0$. This contradicts our choice of $D_s$.

It follows from  linear independence of $D_1$ and $D_{s-1}$  that $\det J(a,c)=[a,c]\neq 0$ (see, for example,  \cite[Corollary 7.2.10]{Nowicki}). Further, we have   $$D_s=[D_{s-1},D_1]=[D_{g(c)},D_{f(a)}]=D_{[g(c),f(a)]}=D_{g'(c)f'(a)[c,a]},$$  and since $D_s=D_{h(a)}$ we get  $$g'(c)f'(a)[c,a]=h(a)+\gamma. \eqno(1)$$ for some $\gamma \in \mathbb K.$
 The field $\mathbb K$ is algebraically closed, so  we have $$h(a)+\gamma=\mu(a-\alpha_1)(a-\alpha_2)\dots(a-\alpha_k),$$  where   $\mu\in \mathbb K^{\star}$ and $\alpha_1,\alpha_2,\dots,\alpha_n$ are all the roots of the polynomial $h(a)+\gamma$. Rewrite the  equality (1)  in the form
$$g'(c)f'(a)[c,a]=\mu(a-\alpha_1)(a-\alpha_2)\dots(a-\alpha_k). \eqno(2).$$
The polynomial $a$ is coordinate and thus   all the  polynomials $a-\alpha_i$, $i=1,2,\dots, k$   are irreducible.

 Let us show    that $[c,a]\in \mathbb K^*$. It was mentioned above that $[c, a]\not =0.$ Assume that $[c, a]\in \mathbb K[x, y]\setminus \mathbb K.$  It follows from (2) that $[c, a]$ is divided by some $a-\alpha _i,$ assume by
 $a-\alpha_1$. Then  $[c, a]=D_c(a)=(a-\alpha_1)u(x,y)$ for some $u(x,y)\in \mathbb K[x,y] .$  It is obvious that
$D_c(a-\alpha_1)=(a-\alpha _1)u(x, y) $  and hence $D_c(a-\alpha _1)=0=[c, a]$ (see, Lemma~\ref{principle_1}(d)). Contradiction.
 Thus $[c,a]\in \mathbb K^{\star}.$  It follows from this relation that $[D_c,D_a]=D_{[c,a]}=0$ and  $D_a$ and $D_c$ are linearly independent over $A.$ By Lemma \ref{coord_pair}, we see that  $(a,c)$ is a coordinate pair for  $A=\mathbb K[x, y].$

 We now find an element $\widetilde {D}$ such that $D_{s-1}$ and $\widetilde {D}$ are linearly independent over $A$ and $[\widetilde {D}, D_{s-1}]=0.$
It follows from the equality~(2) that $g'(c)$ is a nonzero constant, because $\mu(a-\alpha_1)(a-\alpha_2)\dots(a-\alpha_k)$ is divided by $g'(c)$ and the polynomials $a, c$ are algebraically independent over $\mathbb K.$ 
  Thus $g(c)=\beta c +\sigma$ for some $\beta \in \mathbb K^*$, $\sigma \in \mathbb K$ and $D_{s-1}=D_{g(c)}=D_{\beta c+\sigma}$. Without loss of generality, we may assume that $D_{s-1}=D_c$ and $[a,c]=1$. Denote $\deg f(t)$ by $m$ (recall that $D_1=D_{f(a)}$ ). Then put
$$\widetilde{D}=[D_1,\underbrace{D_{s-1},\dots, D_{s-1}}_{m-1}]=D_{f^{(m-1)}(a)[a,c]}\in L,$$ where $f^{(m-1)}(t)$ is the $(m-1)$-th derivative of $f(t)$. Since $\deg f(t)=m$, we get $\widetilde{D}=D_{\delta a+\tau}$ for some $\delta \in \mathbb K^*$, $\tau \in \mathbb K$ and may assume that $\widetilde D=D_a$. Therefore, $D_{s-1}$ and $\widetilde{D}$ are linearly independent over $A$. Moreover, $[\widetilde D,D_{s-1}]=D_{\delta}=0$. Thus by the denoting $D_1=\widetilde D$ and $D_2=D_{s-1}$, we get desired derivations, and $(a,c)$ is the desired coordinate pair in $A$.
\end{proof}
\begin{lemma}\label{automorphism}
Let $L$ be a subalgebra of rank $2$ over $A$ of the Lie algebra $W_2(\mathbb K).$ If $L\subseteq {\rm LND}(A)$  and $\dim _{\mathbb K}L\geq 3,$ then there  exists an automorphism $\theta$ of the ring $A$ such that   \ $\theta L\theta ^{-1}$ contains the elements $ \frac{\partial}{\partial x}, \frac{\partial}{\partial y}, x\frac{\partial}{\partial y}.$
\end{lemma}
\begin{proof}
The Lie algebra $L$ contains  elements $D_a, D_b $ for a coordinate pair $(a, b),$ by Lemma \ref{rank_two}. These elements  are linearly independent  over $A$ and  $[D_a, D_b]=0.$  Define an automorphism $\varphi \in Aut(A)$ by the rule: $\varphi (a)=x,  \varphi (b)=y.$ Then $\varphi$ induces an automorphism $\widetilde \varphi $ of the Lie algebra $W_2(\mathbb K),$ namely:  $\widetilde \varphi (D)=\varphi D\varphi ^{-1}$ for  any $D\in W_2(\mathbb K)$  (see, for example, \cite{Bavula1}). One can easily see that $\varphi D_a\varphi ^{-1}=\frac{\partial}{\partial y}$ and $\varphi D_b\varphi ^{-1}=-\frac{\partial}{\partial x}.$  Denote $L_1=\varphi L\varphi ^{-1}.$ It is a subalgebra of the Lie algebra $W_2(\mathbb K).$ The Lie algebra $L_1$ consists  of locally nilpotent derivations of the ring $A$ and $\frac{\partial}{\partial x}, \frac{\partial}{\partial y}\in L_1.$ Let us show that $L_1$ contains an element $D$ of the form $D=p(x, y)\frac{\partial}{\partial x}+q(x, y)\frac{\partial}{\partial y}$ with $\deg p\leq 1, \deg q\leq 1$ and at least one of these polynomials is nonconstant. Take any element $D_1\in L_1\setminus \mathbb K\langle \frac{\partial}{\partial x}, \frac{\partial}{\partial y}\rangle$ (such an element does exist because $\dim _{\mathbb K}L\geq 3).$  Let $D_1=u(x, y)\frac{\partial}{\partial x}+v(x, y)\frac{\partial}{\partial y},$ where $u(x, y), v(x, y)\in \mathbb K[x, y].$
Without loss of generality we may assume that $\deg u\geq \deg v$ and $\deg u\geq 1.$  Using the following relations
$$[\frac{\partial}{\partial x}, D_1]=u'_x\frac{\partial}{\partial x}+v'_x\frac{\partial}{\partial y}, \ \  [\frac{\partial}{\partial y}, D_1]=u'_y\frac{\partial}{\partial x}+v'_y\frac{\partial}{\partial y},$$
one can easily show that for some $s, k, \ s\geq k$ it holds
$$ \frac{\partial ^{s}u}{\partial x^{k}y^{s-k}}\frac{\partial}{\partial x}+\frac{\partial ^{s}v}{\partial x^{k}y^{s-k}}\frac{\partial}{\partial y}\in L_1,$$ where the polynomial $\frac{\partial ^{s}u}{\partial x^{k}\partial y^{s-k}}$ is of degree $1$ and $\frac{\partial ^{s}v}{\partial x^{k}\partial y^{s-k}}$ is of degree $\leq 1.$ Thus, one may assume that $L_1$ contains an element $D=p(x, y)\frac{\partial}{\partial x}+q(x, y)\frac{\partial}{\partial y},$ where $\deg p\leq 1,$ $\deg q\leq 1$ and $D\not \in \mathbb K\langle \frac{\partial}{\partial x}, \frac{\partial}{\partial y}\rangle.$ Since $\frac{\partial}{\partial x},  \frac{\partial}{\partial y}\in L_1$, this element $D$ can be chosen in the form
$$D=(\alpha _{11}x+\alpha _{12}y)\frac{\partial}{\partial x}+(\alpha _{21}x+\alpha _{22}y)\frac{\partial}{\partial y},$$
where $\alpha _{11},\alpha _{12},\alpha _{21},\alpha _{22}\in \mathbb K$  and at least one of them is nonzero.

The locally nilpotent derivation $D$ has zero divergence (see, for example,  \cite[Corollary 3.16]{Fr}), so   it holds $\alpha _{11}+\alpha _{22}={\rm div}D=0$ . Then $D=(\alpha _{11}x+\alpha _{12}y)\frac{\partial}{\partial x}+(\alpha _{21}x-\alpha _{11}y)\frac{\partial}{\partial y}$ and hence  $D=D_{h}$ for the polynomial $h=\alpha _{21}x^{2}/2-\alpha _{11}xy-\alpha _{12}y^{2}/2.$ There exists (by Lemma \ref{loc_nilp_der}) a coordinate polynomial $c\in A$ such that $h=g(c)$ for some  polynomial $g(t)\in \mathbb K[t].$  If $\deg g=1,$ then $h$ is  a coordinate polynomial. This is impossible, because $h$ is reducible as a homogeneous polynomial in two variables. Hence $\deg g=2$ and $\deg c=1.$ A straightforward check shows that there exist $\mu , \nu \in \mathbb K$ such that $h=(\mu x+\nu y)^{2}.$ Choose a polynomial $\mu _{1}x+\nu _{1}y \ (\mu_1, \nu _1\in \mathbb K)$ in such a way that $\mu \nu _1-\mu _{1}\nu=1.$  The polynomials $\mu x+\nu y, \  \mu_1x+\nu_1$   form a coordinate pair in $\mathbb K[x, y],$ and thus there exists an automorphism $\psi$ of the ring $A$ defined by the rule: $\psi (\mu x+\nu y)=x,  \psi (\mu _{1}x+\nu _{1}y)=y.$  Denote $L_2=\psi L_1\psi ^{-1}.$ One can easily check that
$$\psi D_{\mu x+\nu y}\psi ^{-1}=\frac{\partial}{\partial y}, \  \ \psi D_{\mu _{1}x+\nu _{1}y}\psi ^{-1}=-\frac{\partial}{\partial x}.$$ Since $ D_{\mu x+\nu y}, \  D_{\mu _{1}x+\nu _{1}y}\in L_{1},$ we obtain that $\frac{\partial}{\partial x}, \frac{\partial}{\partial y}\in L_2.$  It follows from the  equality $\psi D_{h}\psi ^{-1}=2x\frac{\partial}{\partial y}$ that $x\frac{\partial}{\partial y}\in L_2.$ Thus $L_2=\theta L\theta ^{-1}$ and $\{ \frac{\partial}{\partial x},  \frac{\partial}{\partial y}, x\frac{\partial}{\partial y}\}\subseteq L_2, $  where $\theta=\psi\varphi \in {\rm Aut}A.$
\end{proof}
\begin{lemma}\label{linearpart}
Let $L$ be a subalgebra of the Lie algebra $W_2(\mathbb K)$ such that $L\subseteq {\rm LND}(A).$ If $\{ \frac{\partial}{\partial x},  \frac{\partial}{\partial y}, x\frac{\partial}{\partial y}\}\subseteq L, $ then every element $D=p(x, y)\frac{\partial}{\partial x}+q(x, y)\frac{\partial}{\partial y}$ of $L$ with ${\rm max}\{\deg p, \deg q\}\leq 1$ belongs to the Lie subalgebra $\mathbb K\langle \frac{\partial}{\partial x},  \frac{\partial}{\partial y}, x\frac{\partial}{\partial y}\rangle.$

\end{lemma}
\begin{proof}
 Since $\frac{\partial}{\partial x},  \frac{\partial}{\partial y}\in L,$ one can assume, without loss of generality, that $D=(\alpha _{11}x+\alpha _{12}y)\frac{\partial}{\partial x}+(\alpha _{21}x+\alpha _{22}y)\frac{\partial}{\partial y}, $ where $\alpha_{11}, \alpha_{12},\alpha_{21},\alpha_{22} \in \mathbb K$ and at least one of them is nonzero.
  The derivation $D$ is locally nilpotent, so we have ${\rm div}D=\alpha _{11}+\alpha _{22}=0$ (see  \cite[Corollary 3.16]{Fr}). Then
$$[x\frac{\partial}{\partial y}, D]=\alpha _{12}(x\frac{\partial}{\partial x}-y\frac{\partial}{\partial y})-2\alpha _{11}x\frac{\partial}{\partial y}\in L.$$ Since $x\frac{\partial}{\partial y}\in L$ we get $\alpha _{12}(x\frac{\partial}{\partial x}-y\frac{\partial}{\partial y})\in L.$ But $(x\frac{\partial}{\partial x}-y\frac{\partial}{\partial y})\not \in {\rm LND}(A)$ and therefore $\alpha _{12}=0.$ Thus $D=\alpha _{11}(x\frac{\partial}{\partial x}-y\frac{\partial}{\partial y})+\alpha _{21}x\frac{\partial}{\partial y}$ and by the same reason, we have $\alpha _{11}=0.$ Hence  $D=\alpha _{21}x\frac{\partial}{\partial y}$ and $D\in \mathbb K\langle \frac{\partial}{\partial x}, \frac{\partial}{\partial y}, x\frac{\partial}{\partial y}\rangle.$
\end{proof}
\begin{lemma}\label{degrees}
Let $L$ be a subalgebra of the Lie algebra $W_2(\mathbb K)$ such that $L\subseteq {\rm LND}(A).$  If $ \{ \frac{\partial}{\partial x},  \frac{\partial}{\partial y}, x\frac{\partial}{\partial y}\}\subseteq L,$ then for any $D=p(x, y)\frac{\partial}{\partial x}+q(x, y)\frac{\partial}{\partial y}\in L$ with ${\rm max}\{ \deg p, \deg q\}\geq 1$ the following is true:

1) $\deg p< \deg q.$

2) the highest homogeneous component of the polynomial $q=q(x, y)$ depends only on $x.$
\end{lemma}
\begin{proof}
Suppose there exists $D=p(x, y)\frac{\partial}{\partial x}+q(x, y)\frac{\partial}{\partial y}\in L$ satisfying conditions of the lemma such that $\deg p\geq \deg q.$ Denote $m=\deg p,$ then $m\geq 1$ by conditions of the lemma.  Since $$[\frac{\partial }{\partial x}, D]=p'_{x}\frac{\partial}{\partial x}+q'_{x}\frac{\partial}{\partial y}\in L \ \mbox{and}  \ [\frac{\partial }{\partial y}, D]=p'_{y}\frac{\partial}{\partial x}+q'_{y}\frac{\partial}{\partial y}\in L,$$
    it is easy to show that for any nonnegative integers $k, s, \ k\leq s,$  it holds
$$ \frac{\partial ^{s}p}{\partial x^{k}\partial y^{s-k}}\frac{\partial}{\partial x}+ \frac{\partial ^{s}q}{\partial x^{k}\partial y^{s-k}}\frac{\partial}{\partial y}\in L.$$
Denote by $p_{m}(x, y)$ the highest homogeneous component  of the polynomial $p(x, y.$  Let $p_{m}(x, y)=\sum _{i=0}^{m}\alpha _{i, m-i}x^{i}y^{m-i}$ for $ \alpha _{ij}\in \mathbb K$ and let, for example, $\alpha_{k,m-k}\neq 0.$ 
First, let $k>0.$ Then as above
$$D_1:= \frac{\partial ^{m-1}p}{\partial x^{k-1}\partial y^{m-k}}\frac{\partial}{\partial x}+ \frac{\partial ^{m-1}q}{\partial x^{k-1}\partial y^{m-k}}\frac{\partial}{\partial y}\in L,$$ and $D_1$ is of the form $D_1=(\alpha _{1}x+\beta _{1}y+\gamma _1)\frac{\partial}{\partial x}+(\delta _{1}x+\mu _{1}y+\nu _1)\frac{\partial}{\partial y}$ with all the coefficients in $\mathbb K.$ Since $\alpha _{k, m-k}\not =0,$  we have  $\alpha _1\not =0.$  The latter is impossible by Lemma \ref{linearpart}. Further, if $k=0,$  i.e. $\alpha _{0, m}\not =0,$  we get $$D_2:= \frac{\partial ^{m-1}p}{\partial y^{m-1}}\frac{\partial}{\partial x}+ \frac{\partial ^{m-1}q}{\partial y^{m-1}}\frac{\partial}{\partial y}\in L,$$
 $D_2$ is of the form $D_2=(\alpha _{2}x+\beta _{2}y+\gamma _2)\frac{\partial}{\partial x}+(\delta _{2}x+\mu _{2}y+\nu _2)\frac{\partial}{\partial y}$ with all the coefficients in $\mathbb K,$ and $\alpha _2\not =0$ because $\alpha _{0, m}\not =0.$  This is also impossible by the same reason. Therefore $\deg p<\deg q$ for any derivation $D=p(x, y)\frac{\partial}{\partial x}+q(x, y)\frac{\partial}{\partial y}\in L.$

Denote $n=\deg q(x, y)$ and let $q_n(x, y)$ be the highest homogeneous component  of the polynomial $q.$ Suppose $\deg _{y}q_{n}(x, y)=l\geq 1.$ Then as above $$D_3:= \frac{\partial ^{n-1}p}{\partial x^{n-l}\partial y^{l-1}}\frac{\partial}{\partial x}+ \frac{\partial ^{n-1}q}{\partial x^{n-l}\partial y^{l-1}}\frac{\partial}{\partial y}\in L$$
and $D_3$ is of the form  $D_3=\alpha \frac{\partial}{\partial x}+(\beta x+\gamma y+\delta)\frac{\partial}{\partial y}$
with $\gamma \not =0$ because $\deg _{y}q_n=l.$
 Since $  \{ \frac{\partial}{\partial x},  \frac{\partial}{\partial y}, x\frac{\partial}{\partial y}\}\subseteq L $ we get $y\frac{\partial}{\partial y}\in L.$ The latter is impossible because $y\frac{\partial}{\partial y}\not \in {\rm LND}(A).$
The obtained contradiction shows that $\deg _{y}q_n(x, y)=0$ and therefore $q_n=q_n(x).$
\end{proof}
\begin{lemma}\label{final}
Let $L$ be a subalgebra of the Lie algebra $W_2(\mathbb K)$ such that $L\subseteq {\rm LND}(A).$ If $  \{ \frac{\partial}{\partial x},  \frac{\partial}{\partial y}, x\frac{\partial}{\partial y}\}\subseteq L, $ then every element $D\in L$ is of the form $D=\alpha \frac{\partial}{\partial x}+q(x)\frac{\partial}{\partial y}, $ where $\alpha \in \mathbb K$ and $q(t)\in \mathbb K[t].$

\end{lemma}

\begin{proof}
Suppose $L$ contains elements $D$ of the form $D=p(x, y)\frac{\partial}{\partial x}+q(x, y)\frac{\partial}{\partial y}$ with $p(x, y)\in \mathbb K[x, y]\setminus \mathbb K.$ Choose among such elements an element ${ D}={ p}(x, y)\frac{\partial}{\partial x}+{q}(x, y)\frac{\partial}{\partial y}$ with a minimum $\deg {q}.$ Let us show that $ { p}(x, y)$ is a polynomial only in $x.$  Suppose to the  contrary that ${p}'_{y}(x, y)\not =0.$ Then
$$D_{1}:=[x\frac{\partial}{\partial y}, { D}]=[x\frac{\partial}{\partial y}, { p}\frac{\partial}{\partial x}+{q}\frac{\partial}{\partial y}]=x{p}'_{y}\frac{\partial}{\partial x}+(-{p}+x{ q}'_{y})\frac{\partial}{\partial y}$$
and $D_1\in L.$ Since $x{ p}'_{y}\not = {\rm const},$ we have that $\deg (-{ p}+x{q}'_{y})\geq \deg {q}$ by the choice of the polynomial $ q.$
By Lemma \ref{degrees}(1), $\deg p<\deg q,$ and by Lemma \ref{degrees}(2), $\deg q'_y\leq \deg q-2.$ Thus $\deg(-p+xq'_y)<\deg q.$
 This contradicts  our choice of $D$ and therefore  ${p}'_{y}=0,$ i.e. $p=p(x).$

 Further $$[\frac{\partial}{\partial x}, { D}]={ p}'_{x}\frac{\partial}{\partial x}+{q}'_{x}(x, y)\frac{\partial}{\partial y}\in L$$
 and $ \deg {q}'_{x}< \deg {q}.$ By the choice of $D,$ we get that ${ p}'_{x}\in \mathbb K.$ Since ${p}(x)\in \mathbb K[x, y]\setminus \mathbb K,$  it holds ${p}(x)=\alpha x+\beta$ for some $\alpha \in \mathbb K^{\star}, \beta \in \mathbb K.$ Without loss of generality, one can assume that $\beta =0$ because $\frac{\partial}{\partial x}\in L.$ We have ${D}=\alpha x\frac{\partial}{\partial x}+{q}(x, y)\frac{\partial}{\partial y}.$ Then
 $$[\frac{\partial}{\partial x}, { D}]=\alpha \frac{\partial}{\partial x}+{q}'_{x}(x, y)\frac{\partial}{\partial y}\in L$$ and hence ${q}'_{x}(x, y)\frac{\partial}{\partial y}\in L.$ The inclusion $L\subseteq {\rm LND}(A)$ implies that every element of $L$ has zero divergence and therefore ${q}''_{xy}(x, y)=0.$
One can easily show  that ${q}(x, y)=u(y)+r(x)$ for some  univariate polynomials $u(t), r(t)\in \mathbb K[t].$ The latter means that ${D}=\alpha x\frac{\partial}{\partial x}+(u(y)+r(x))\frac{\partial}{\partial y}$ and since ${\rm div}{ D}=\alpha +u'(y)=0,$ we get $u(y)=-\alpha  y+\delta$ for some $ \delta \in \mathbb K.$ Since $\frac{\partial}{\partial y}\in L,$ we may assume  that $\delta =0. $  Thus
${D}=\alpha (x\frac{\partial}{\partial x}-y\frac{\partial}{\partial y})+r(x)\frac{\partial}{\partial y}.$
By Lemma \ref{loc_nilp_der}, ${D}=D_{-\alpha xy+s(x)},$  where $s(x)$ is  a polynomial such that $s'(x)=r(x).$ By the same lemma, we have  $-\alpha xy +s(x)=f(a)$ for a coordinate polynomial $a=a(x, y)$ and some $f(t)\in \mathbb K[t].$  Note that $a'_{y}(x, y)\not =0$ because in the other case $\alpha =0$  which contradicts the hypothesis on $\alpha$ (recall $\alpha \in \mathbb K^*$).

Since this fact, if $\deg f(t)\geq 2,$ then we get $\deg _{y}f(a)\geq 2.$    The latter is impossible because of equality $f(a)=-\alpha xy+s(x).$ Therefore, $\deg f(t)=1$  and $f(a)=\alpha xy+s(x)$ is a coordinate polynomial of the ring $\mathbb K[x, y].$ Denote by $s_{0}$  the constant term of the polynomial $s(x).$ We see that $f(a)-s_0=-\alpha xy-s(x)-s_0$ is also a coordinate polynomial. But the polynomial $\alpha xy-s(x)-s_0$ divides by $x$ and thus  is  reducible. The obtained contradiction shows that every element $D\in L$ is of the form $D=\alpha \frac{\partial}{\partial x}+q(x, y)\frac{\partial}{\partial y}, \alpha \in \mathbb K.$ Since ${\rm div} D=q'_{y}(x,y)=0$, we get that $q$ depends only on $x$. Therefore, $D=\alpha \frac{\partial}{\partial x}+q(x)\frac{\partial}{\partial y}$.

\end{proof}
\begin{theorem}\label{Th2}
Let $L$ be a subalgebra of the Lie algebra $W_2(\mathbb K).$  If $L$ consists of locally nilpotent derivations of the ring $\mathbb K[x, y],$ then there exists an automorphism $\varphi :\mathbb K[x, y]\to \mathbb K[x, y]$ such that  $L_1=\varphi L\varphi ^{-1}$ is a subalgebra of the triangular Lie algebra $u_{2}(\mathbb K).$
\end{theorem}
\begin{proof}
If $L$ is abelian, then the statement of the  theorem follows from Lemma \ref{abelian}. Let $L$ be nonabelian. Then ${\rm rk}_{A}(L)=2$ and $\dim _{\mathbb K}L\geq 3.$ By Lemma \ref{automorphism}, there exists  an automorphism $\varphi$ of the polynomial ring $\mathbb K[x, y]$ such that the Lie algebra $L_1=\varphi L\varphi ^{-1}$ contains the elements $\frac{\partial}{\partial x}, \frac{\partial}{\partial y}, x\frac{\partial}{\partial y}.$ By Lemma \ref{final},  every element $D\in L_1$ is of the form $ D=\alpha \frac{\partial}{\partial x}+q(x)\frac{\partial}{\partial y}.$  The latter means that $L_1\subseteq u_2(\mathbb K).$
\end{proof}

\begin{corollary}
Every maximal (by inclusion) subalgebra of the Lie algebra $W_2(\mathbb K)$ which is contained in the  set ${\rm LND}(\mathbb K[x, y])$  is either $u_2(\mathbb K)$ or one of its conjugated by  automorphisms of $\mathbb K[x, y]$ subalgebras.
\end{corollary}

\begin{remark}
By the known Miyanishi's Theorem \cite{Miya} and a result of D. Daigle \cite{Daigle}, every nonzero $D\in {\rm LND}(\mathbb K[x, y, z])$ is of the form $D=hD_{(a, b)},$  where $a,b \in \mathbb K[x,y,z]$ such that $\Ker D=\mathbb K[a,b]$,  $D_{(a, b)}$ is a jacobian derivation and $h\in \Ker D=\mathbb K[a, b].$
Unfortunately, this fact does not enable  to prove that every Lie algebra $L\subseteq  {\rm LND}(\mathbb K[x, y, z])$ is Engelian (as in the case ${\rm LND}(\mathbb K[x, y])$, in Corollary \ref{Engel}).
\end{remark}

%%%%%%%%%%%%%%%%%%%%%%%%%%%%%%%%%%%%%%%%%%

%
\end{document}